\documentclass{amsart}
\usepackage{amssymb,mathtools}
\usepackage[british]{babel}
\usepackage{enumitem}
\usepackage[latin1]{inputenc}

\usepackage{url}
\usepackage{xcolor}

\newtheorem{theorem}{Theorem}
\numberwithin{theorem}{section}
\newtheorem{corollary}[theorem]{Corollary}
\newtheorem{lemma}[theorem]{Lemma}
\newtheorem{proposition}[theorem]{Proposition}
\theoremstyle{definition}
\newtheorem{definition}[theorem]{Definition}
\newtheorem{remark}[theorem]{Remark}

\numberwithin{equation}{section}

\newcommand{\po}{\mathsf{PO}}
\newcommand{\wpo}{\mathsf{WPO}}
\newcommand{\fin}[1]{\left[#1\right]^{<\omega}}
\newcommand{\To}{\Rightarrow}
\newcommand{\rng}{\operatorname{rng}}
\newcommand{\supp}{\operatorname{supp}}
\newcommand{\en}{\operatorname{en}}
\newcommand{\nf}{\mathrel{=_{\operatorname{nf}}}}
\newcommand{\tr}{\operatorname{Tr}}
\newcommand{\rca}{\mathsf{RCA}}
\newcommand{\aca}{\mathsf{ACA}}
\newcommand{\ads}{\mathsf{ADS}}
\newcommand{\cac}{\mathsf{CAC}}
\newcommand{\len}{\operatorname{len}}
\newcommand{\tl}{\trianglelefteq}

\newcommand{\period}{\text{.}}

\newcommand{\tw}{\mathcal T W}

\begin{document}
	
\title[The uniform Kruskal theorem]{The uniform Kruskal theorem: between finite combinatorics and strong set existence}
\author{Anton Freund and Patrick Uftring}

\address{Department of Mathematics, Technical University of Darmstadt, Schloss\-garten\-str.~7, 64289~Darmstadt, Germany}
\email{\{freund,uftring\}@mathematik.tu-darmstadt.de}

\thanks{Funded by the Deutsche Forschungsgemeinschaft (DFG, German Research Foundation) -- Project number 460597863.}

\begin{abstract}
The uniform Kruskal theorem extends the original result for trees to general recursive data types. As shown by A.~Freund, M.~Rathjen and A.~Weier\-mann, it is equivalent to $\Pi^1_1$-comprehension, over $\mathsf{RCA_0}$ with the chain antichain principle~($\mathsf{CAC}$). This result provides a connection between finite combinatorics and abstract set existence. The present paper sheds further light on this connection. First, we show that the original Kruskal theorem is equivalent to the uniform version for data types that are finitely generated. Secondly, we prove a dichotomy result for a natural variant of the uniform Kruskal theorem. On the one hand, this variant still implies \mbox{$\Pi^1_1$-}comprehension over $\mathsf{RCA}_0+\mathsf{CAC}$. On the other hand, it becomes weak when $\mathsf{CAC}$ is removed from the base theory.
\end{abstract}


\maketitle

\section{Introduction}

Kruskal's theorem~\cite{kruskal60} asserts that any infinite sequence $t_0,t_1,\ldots$ of finite rooted trees admits \mbox{$i<j$} such that $t_i$ embeds into~$t_j$. Precise definitions for the case of ordered trees can be found in Section~\ref{sect:fin-uniform-original} below. The result for unordered trees has the same strength, due to results by M.~Rathjen and A.~Weier\-mann~\cite{rathjen-weiermann-kruskal} as well as J.~van der Meeren~\cite[Chapter~2]{meeren-thesis}.

As famously shown by H.~Friedman (see~\cite{simpson85}), Kruskal's theorem is unprovable in predicative axiom systems, so that we have a concrete mathematical example for the incompleteness phenomenon from G\"odel's theorems. Specifically, Kruskal's theorem has consistency strength strictly between $\mathsf{ATR_0}$ and $\Pi^1_1\text{-}\mathsf{CA_0}$, the two strongest of five particularly prominent axiom systems from reverse mathematics (see the textbook by S.~Simpson~\cite{simpson09} for general background).

In work of Freund, Rathjen and Weiermann~\cite{frw-kruskal}, it is shown that $\Pi^1_1$-compre\-hension (the~central axiom of $\Pi^1_1\text{-}\mathsf{CA_0}$) is equivalent to a uniform version of Kruskal's theorem, which extends the latter from trees to general recursive data types. In our view, the equivalence sheds some light on a fascinating phenomenon that is typical for modern mathematics, namely on connections between the abstract and the concrete. To refine our understanding of these connections, the present paper investigates how the strength of the uniform Kruskal theorem is affected by certain finiteness conditions.

Our first aim is to recall the precise formulation of the uniform Kruskal theorem. As the following discussion is rather condensed, we point out that more explanations and examples can be found in the introduction of~\cite{frw-kruskal}. A function $f:X\to Y$ between partial orders is called a quasi embedding if it reflects the order, i.\,e., if $f(x)\leq_Y f(x')$ entails~$x\leq_X x'$. It is an embedding if the order is also preserved, i.\,e., if the converse implication holds as well. By $\mathsf{PO}$ we denote the category with the partial orders as objects and the quasi embeddings as morphisms. A functor $W:\po\to\po$ is said to preserve embeddings if~$W(f)$ is an embedding whenever the same holds for~$f$. For a set~$X$ and a function~$f:X\to Y$, we declare
\begin{align*}
\fin{X}&:=\text{``the set of finite subsets of~$X$"},\\
\fin{f}(a)&:=\{f(x)\,|\,x\in a\}\in\fin{Y}\quad\text{for}\quad a\in\fin{X},
\end{align*}
which yields an endofunctor of sets. The forgetful functor from partial orders to sets will be left implicit. The following adapts a classical notion of J.-Y.~Girard~\cite{girard-pi2} from linear to partial orders. We write $\rng(f)$ for the range or image of a function~$f$.

\begin{definition}\label{def:po-dilator}
A $\po$-dilator consists of a functor $W:\po\to\po$ that preserves embeddings and a natural transformation $\supp:W\To\fin{\cdot}$ such that we have
\begin{equation*}
\supp_Y(\sigma)\subseteq\rng(f)\quad\Rightarrow\quad\sigma\in\rng(W(f))
\end{equation*}
for any embedding~$f:X\to Y$ and any~$\sigma\in W(Y)$. If $W(X)$ is a well partial order whenever the same holds for~$X$, then $W$ is called a $\wpo$-dilator. We say that $W$ is normal if
\begin{equation*}
\sigma\leq_{W(X)}\tau\ \ \To\ \ \text{for any }x\in\supp_X(\sigma)\text{ there is an }x'\in\supp_X(\tau)\text{ with }x\leq_X x'
\end{equation*}
holds for any partial order~$X$ and all~$\sigma,\tau\in W(X)$.
\end{definition}

One may think of $\sigma\in W(X)$ as a finite structure with labels from~\mbox{$\supp_X(\sigma)\subseteq X$}. In typical examples, we have $\sigma\leq_{W(X)}\tau$ if there is an embedding that sends each label to a larger one, which witnesses that~$W$ is normal.

\begin{definition}\label{def:Kruskal-fp}
A Kruskal fixed point of a normal~$\po$-dilator~$W$ consists of a partial order~$X$ and a bijection $\kappa:W(X)\to X$ such that all $\sigma,\tau\in X$ validate
\begin{equation*}
\kappa(\sigma)\leq_X\kappa(\tau)\quad\Leftrightarrow\quad\sigma\leq_{W(X)}\tau\text{ or }\kappa(\sigma)\leq_X x\text{ for some }x\in\supp_X(\tau).
\end{equation*}
We say that the given Kruskal fixed point is initial if any other Kruskal fixed point $\pi:W(Y)\to Y$ admits a unique quasi embedding $f:X\to Y$ with $f\circ\kappa=\pi\circ W(f)$.
\end{definition}

The trees form the initial Kruskal fixed point of the transformation that maps~$X$ to the collection of finite sequences in~$X$ (with the order from Higman's lemma). Here $\kappa$ implements the usual recursive construction that builds a tree $\kappa(\sigma)$ from the list~$\sigma$ of its immediate subtrees. The equivalence above says that a tree $\kappa(\sigma)$ embeds into~$\kappa(\tau)$ when the immediate subtrees~$\sigma$ embed into corresponding subtrees~$\tau$ or the entire tree $\kappa(\sigma)$ embeds into some subtree~$x\in\supp_X(\tau)$. In~\cite[Section~3]{frw-kruskal} it is shown that each normal $\po$-dilator~$W$ has an initial Kruskal fixed point, which is unique up to isomorphism and will be denoted by~$\mathcal T W$.

\begin{definition}
By the uniform Kruskal theorem for a collection $\Gamma$ of $\po$-dilators, we mean the statement that $\mathcal T W$ is a well partial order for every normal $W\in\Gamma$.
\end{definition}

Let us stress that normality is assumed by default (and will not be mentioned explicitly), as the very construction of~$\mathcal T W$ may fail when $W$ is not normal. The uniform Kruskal theorem is true when $\Gamma$ is the collection of all $\wpo$-dilators, which is the case studied in~\cite{frw-kruskal}. When $W$ is no \mbox{$\wpo$-}dilator, $\mathcal T W$ may fail to be a well partial order (consider constant~$W$ with empty supports). To get the original Kruskal theorem for trees, one declares that $\Gamma$ contains the single $\wpo$-dilator that maps~$X$ to the sequences in~$X$.

We now recall a more concrete approach to $\po$-dilators, which allows for a formalization in reverse mathematics. Let $\po_0\subseteq\po$ contain a unique representative from each isomorphism class of finite partial orders. For each such order~$a$, we pick an isomorphism $\en_a:|a|\to a$ with~$|a|\in\po_0$. When $a$ is a finite suborder of some~$X$, we define $\en_a^X:=\iota\circ\en_a$ with $\iota:a\hookrightarrow X$. Due to the support condition from Definition~\ref{def:po-dilator}, each $\sigma\in W(X)$ has a unique normal form
\begin{equation*}
\sigma\nf W(\en_a^X)(\sigma_0)\quad\text{with}\quad a=\supp_X(\sigma),
\end{equation*}
where we have $\sigma_0\in W(|a|)$. By the naturality of supports, we get
\begin{equation*}
a=\supp_X(\sigma)=[\en_a^X]^{<\omega}\circ\supp_{|a|}(\sigma_0)\quad\Leftrightarrow\quad |a|=\supp_{|a|}(\sigma_0).
\end{equation*}
This motivates the following analogue of a definition due to Girard~\cite{girard-pi2}.

\begin{definition}\label{def:trace}
The trace of a $\po$-dilator~$W$ is given by
\begin{equation*}
\tr(W):=\{(c,\rho)\,|\,c\in\po_0\text{ and }\rho\in W(c)\text{ with }\supp_c(\rho)=c\}.
\end{equation*}
We say that~$W$ is finite if the same holds for~$\tr(W)$.
\end{definition}

By the above, the elements of~$W(X)$ can be uniquely represented by pairs $(a,\sigma_0)$ with \mbox{$a\in[X]^{<\omega}$} and $(|a|,\sigma_0)\in\tr(W)$. One can infer that~$W$ is determined (up to isomorphism) by its restriction to~$\po_0$. Assuming that $\tr(W)$ is countable, this makes it possible to represent~$W$ in the framework of reverse mathematics (see~\cite{frw-kruskal} for details). Now consider an initial Kruskal fixed point $\kappa:W(\mathcal T W)\to\mathcal T W$. Given $a\in[\mathcal T W]^{<\omega}$ and $(|a|,\sigma_0)\in\tr(W)$, we abbreviate
\begin{equation*}
\circ(a,\sigma_0):=\kappa(\sigma)\quad\text{for}\quad\sigma\nf W(\en_a^{\mathcal T W})(\sigma_0).
\end{equation*}
Since $\kappa$ is bijective, each element of~$\mathcal T W$ can be uniquely written as~$\circ(a,\sigma_0)$. Also, the fact that $\kappa$ is initial ensures that we get a height function
\begin{equation*}
h:\mathcal T W\to\mathbb N\quad\text{with}\quad h\big(\circ(a,\sigma_0)\big)=\max\big(\{0\}\cup\{h(t)+1\,|\,t\in a\}\big).
\end{equation*}
This suggests a recursive construction of $\mathcal T W$ as a set of terms, which is available in the weak base theory~$\mathsf{RCA_0}$ from reverse mathematics. The latter shows that $\mathcal T W$ is an initial Kruskal fixed point, as verified in~\cite[Section~3]{frw-kruskal}. At the same time, the principle that initial Kruskal fixed points are well partial orders is very strong, by the following theorem that was already mentioned above. We recall~that the ascending descending sequence principle ($\ads$) asserts that every infinite linear order contains an infinite sequence that is either strictly increasing or strictly decreasing. In~\cite{frw-kruskal}, the theorem was stated with the chain antichain principle ($\mathsf{CAC}$) at the place of~$\mathsf{ADS}$. The latter is weaker (see~\cite{lerman-solomon-towsner}) and suffices according to Remark~\ref{rmk:ads} below.

\begin{theorem}[\cite{frw-kruskal}]\label{thm:unif-kruskal}
The following are equivalent over $\mathsf{RCA_0}+\ads$:
\begin{enumerate}[label=(\roman*)]
\item the uniform Kruskal theorem for $\wpo$-dilators,
\item the principle of $\Pi^1_1$-comprehension.
\end{enumerate}
\end{theorem}

In the first half of the present paper, we study the uniform Kruskal theorem for $\po$-dilators that are finite. The initial Kruskal fixed points of these $\po$-dilators may be seen as finitely generated data types, since we can view $\circ(a,\sigma_0)\in\mathcal T W$ as the value of the constructor~$\sigma_0$ on arguments~$a$. To summarize our results, we introduce one further notion:

\begin{definition}
For quasi embeddings $f,g:X\to Y$ between partial orders, we write $f\leq g$ if we have $f(x)\leq_Y g(x)$ for all~$x\in X$. A $\po$-dilator~$W$ is called monotone if $f\leq g$ entails $W(f)\leq W(g)$.
\end{definition}

As we show in the next section, each $\wpo$-dilator is monotone, provably in~$\rca_0$. Conversely, any finite monotone $\po$-dilator is a $\wpo$-dilator, and this is equivalent to the statement that $X\mapsto X^n$ preserves well partial orders for all numbers $n\in\mathbb N$. Essentially over~$\rca_0$ (but modulo a result that Rathjen and Weiermann~\cite{rathjen-weiermann-kruskal} prove over~$\mathsf{ACA_0}$), we show that the uniform Kruskal theorem for $\po$-dilators that are both finite and monotone is equivalent to the original Kruskal theorem for trees (which is considerably weaker than the statements from Theorem~\ref{thm:unif-kruskal}). The proof idea comes~from the final section of Kruskal's original paper~\cite{kruskal60}, but one needs to show that it applies in the general setting of dilators. This is not self-evident, as related statements remain strong under finiteness conditions (consider binary trees with gap condition in~\cite{MRW-Bachmann}, weakly finite dilators in~\cite{apw-fast-growing} and the paragraph after Corollary~\ref{cor:unif-fin-trees}~below).

In the second half of our paper, we investigate versions of the uniform Kruskal theorem in the absence of~$\ads$. Let us write $\psi$ for the uniform Kruskal theorem for \mbox{$\wpo$-}dilators~$W$ such that $W(1)$ is finite (where $1$ is the order with a single element). As we show in Section~\ref{sect:without-ADS}, this variant leads to the following surprising dichotomy: On the one hand, $\psi$ does still imply $\Pi^1_1$-comprehension over~$\rca_0+\cac$. On the other hand, any \mbox{$\Pi^1_2$-}theorem of $\rca_0+\psi$ can already be proved in a theory that is weaker than $\aca_0$ (and hence much weaker than $\Pi^1_1$-comprehension). As a consequence, $\rca_0$~cannot prove that the original Kruskal theorem for binary trees follows from the uniform Kruskal theorem for finite $\wpo$-dilators, in contrast with the previous paragraph (note that the finite $\wpo$-dilators and the finite monotone $\po$-dilators do not coincide over $\rca_0$). We do not know whether a similar dichotomy holds in the case of Theorem~\ref{thm:unif-kruskal}. At any rate, it seems remarkable that some natural version of the uniform Kruskal theorem can have either predicative or impredicative strength, depending on weak principles in the base theory.

\subsection*{Acknowledgements} A variant of Theorem~\ref{thm:fin-uniform-original} had been proved in joint work with Michael Rathjen and Andreas Weiermann~\cite{frw-ccc-talk}. We are grateful that the latter have allowed us to adapt the result for the present paper. We also thank Gabriele Buriola and Peter Schuster, who pursue a related project and have pointed out the original formulation of a result due to G.~Higman (cf.~the paragraph before Corollary~\ref{cor:unif-fin-trees}). Furthermore, we thank Leszek Ko{\l}odziejczyk for helpful suggestions on the reverse mathematics of~$\ads$.

\section{Trees and finitely generated data types}\label{sect:fin-uniform-original}

In this section, we prove an equivalence between the usual Kruskal theorem for trees and the uniform Kruskal theorem for finite monotone $\po$-dilators. Before, we show that a finite $\po$-dilator is monotone precisely if it is a $\wpo$-dilator.

The first implication holds over the usual weak base theory. We note that Girard has proved the analogous result for dilators on linear orders, using the assumption that $\omega^\omega$ is well founded (see~\cite[Proposition~2.3.10]{girard-pi2} and also~\cite[Lemma~5.3]{frw-kruskal}).

\begin{proposition}\label{prop:wpo-monotone}
Any $\wpo$-dilator is monotone, provably in~$\rca_0$.
\end{proposition}
\begin{proof}
We identify $n\in\mathbb N$ with the discrete partial order $\{0,\ldots,n-1\}$ and equip $n\times\mathbb N$ with the usual product order, so that $(c,k)\leq(d,l)$ holds precisely when we have $c=d$ and $k\leq l$ as natural numbers. To see that $\rca_0$ recognizes $n\times\mathbb N$ as a well partial order, we consider an infinite sequence
\begin{equation*}
(c_0,k_0),(c_1,k_1),\ldots\subseteq n\times\mathbb N
\end{equation*}
and assume that it is bad, i.\,e., that $(c_i,k_i)\not\leq(c_j,k_j)$ holds for all~$i<j$. For fixed~$i$, there are only finitely many $(c_j,k_j)$ with $k_j<k_i$. After passing to a~recursive sub\-sequence, we may thus assume that $i\mapsto k_i$ is increasing. But we easily find indices $i<j$ with $c_i=c_j$. This yields ${(c_i,k_i)\leq (c_j,k_j)}$, against our assumption. Given a $\wpo$-dilator~$W$, we conclude that $W(n\times\mathbb N)$ is a well partial order for all~$n\in\mathbb N$. To deduce that $W$ is monotone, we consider two quasi embeddings ${f,g:X\to Y}$ with~$f\leq g$. We first prove $W(f)\leq W(g)$ under the following assumptions:
\begin{enumerate}[label=(A\arabic*)]
\item the partial order~$X$ is finite,
\item any element of~$Y$ lies in the range of~$f$ or~$g$,
\item we have $f(x)\neq g(x')$ for $x\neq x'$.
\end{enumerate}
In view of~(A1) we fix an enumeration of the $n$-element set $X=\{x_0,\ldots,x_{n-1}\}$. For each $k\in\mathbb N$ we have a quasi embedding
\begin{equation*}
h_k:X\to n\times\mathbb N\quad\text{with}\quad h_k(x_i):=\begin{cases}
(i,0) & \text{if }f(x_i)=g(x_i),\\
(i,k) & \text{otherwise}.
\end{cases}
\end{equation*}
Given $\sigma\in W(X)$, we find $k<l$ with $W(h_k)(\sigma)\leq W(h_l)(\sigma)$ in $W(n\times\mathbb N)$, as the latter is a well partial order. Let us recall that any quasi embedding is injective (use antisymmetry in~$X$ to derive $x=x'$ from $f(x)=f(x')$). Combined with conditions (A2) and~(A3), this allows us to consider
\begin{equation*}
h:Y\to n\times\mathbb N\quad\text{with}\quad h(y):=\begin{cases}
(i,0) & \text{if }y=f(x_i)=g(x_i),\\
(i,k) & \text{if }y=f(x_i)\neq g(x_i),\\
(i,l) & \text{if }y=g(x_i)\neq f(x_i).
\end{cases}
\end{equation*}
To show that $h$ is a quasi embedding, we point out that $h(y)=(i,k)\leq (i,l)=h(y')$ entails $y=f(x_i)\leq g(x_i)=y'$, where the inequality relies on~$f\leq g$. Now observe
\begin{equation*}
W(h\circ f)(\sigma)=W(h_k)(\sigma)\leq W(h_l)(\sigma)=W(h\circ g)(\sigma).
\end{equation*}
Since $W(h)$ is a quasi embedding, we get $W(f)(\sigma)\leq W(g)(\sigma)$, as required. Let us now consider quasi embeddings $f,g:X\to Y$ with $f\leq g$ that validate~(A3) but not necessarily (A1) and~(A2). Given $\sigma\in W(X)$, we set
\begin{equation*}
X_0:=\supp_X(\sigma)\subseteq X\quad\text{and}\quad Y_0:=[f]^{<\omega}(X_0)\cup[g]^{<\omega}(X_0)\subseteq Y.
\end{equation*}
Write~$\iota_X:X_0\hookrightarrow X$ and $\iota_Y:Y_0\to Y$ for the inclusions. We define ${f_0,g_0:X_0\to Y_0}$ by stipulating $\iota_Y\circ f_0=f\circ\iota_X$ and $\iota_Y\circ g_0=g\circ\iota_X$. This yields quasi embeddings ${f_0\leq g_0}$ that validate (A1) to~(A3). Due to the support condition from Definition~\ref{def:po-dilator}, we may write $\sigma=W(\iota_X)(\sigma_0)$ with $\sigma_0\in W(X_0)$. By the result for $f_0,g_0$ and the fact that $W(\iota_Y)$ is an embedding (also by Definition~\ref{def:po-dilator}), we get
\begin{equation*}
W(f)(\sigma)=W(f\circ\iota_X)(\sigma_0)=W(\iota_Y\circ f_0)(\sigma_0)\leq W(\iota_Y\circ g_0)(\sigma_0)=W(g)(\sigma).
\end{equation*}
Finally, we also drop~(A3) and consider arbitrary quasi embeddings $f,g:X\to Y$ with ${f\leq g}$. Let us equip $Y\times 2$ with the lexicographic order, in which we have $(y,0)<(y,1)$ and ${(y,i)<(y',j)}$ whenever $y<y'$ holds in~$Y$ (note that~$2$ is no longer discrete). We consider the embedding~${\iota:Y\to Y\times 2}$ with $\iota(y):=(y,0)$ and the quasi embedding
\begin{equation*}
f^+:X\to Y\times 2\quad\text{with}\quad f^+(x):=\begin{cases}
(f(x),0) & \text{if }f(x)=g(x),\\
(f(x),1) & \text{otherwise}.
\end{cases}
\end{equation*}
Note that we have $\iota\circ f\leq f^+$ and that~(A3) holds for $\iota\circ f$ and $f^+$. To see that we have $f^+\leq\iota\circ g$, assume $f(x)\neq g(x)$ and note that $f\leq g$ forces $f(x)<g(x)$, so that we get
\begin{equation*}
f^+(x)=(f(x),1)<(g(x),0)=\iota\circ g(x).
\end{equation*}
Furthermore, $f^+$ and $\iota\circ g$ validate~(A3). Indeed, if we have
\begin{equation*}
f^+(x)=(f(x),i)=(g(x'),0)=\iota\circ g(x'),
\end{equation*}
we get $i=0$ and hence $f(x)=g(x)$, which yields~$x=x'$ by the injectivity of~$g$. For $\sigma\in W(X)$, the result under~(A3) will now yield
\begin{equation*}
W(\iota\circ f)(\sigma)\leq W(f^+)(\sigma)\leq W(\iota\circ g)(\sigma).
\end{equation*}
As $W(\iota)$ is a (quasi) embedding, we once again get $W(f)(\sigma)\leq W(g)(\sigma)$.
\end{proof}

For the case of dilators on linear orders, Girard has also proved a converse result (see~\cite[Proposition~4.3.8]{girard-pi2}). We establish the analogue for partial orders and include a reversal. Let us point out that $X^n$ carries the usual product order, in which we have $\langle x_0,\ldots,x_{n-1}\rangle\leq\langle x'_0,\ldots,x'_{n-1}\rangle$ precisely when $x_i\leq x'_i$ holds for all $i<n$. The following statement~(ii) is known to entail both $\ads$ and $\Sigma^0_2$-induction (see the remark before Theorem~2.22 of~\cite{marcone-wqo-bqo} and Theorem~5 of~\cite{uftring-etr}). 

\begin{proposition}\label{prop:monotone-to-wpo}
The following are equivalent over $\rca_0$:
\begin{enumerate}[label=(\roman*)]
\item any finite monotone $\po$-dilator is a $\wpo$-dilator,
\item if $X$ is a well partial order, then so is $X^n$ for every~$n\in\mathbb N$.
\end{enumerate}
\end{proposition}
\begin{proof}
To see that~(i) implies~(ii), we consider the monotone $\po$-dilator~$W$ with $W(X):=X^n$ and
\begin{align*}
W(f)(\langle x_0,\ldots,x_{n-1}\rangle)&:=\langle f(x_0),\ldots,f(x_{n-1})\rangle,\\
\supp_X(\langle x_0,\ldots,x_{n-1}\rangle)&:=\{x_0,\ldots,x_{n-1}\}.
\end{align*}
It suffices to show that (the trace of)~$W$ is finite. Let us recall that $(c,\rho)\in\tr(W)$ amounts to $\rho\in W(c)$ and $\supp_c(\rho)=c\in\po_0$. The latter entails that $c$ can have at most $n$ elements. But then there are only finitely many possibilities for~$c$, since $\po_0$ contains a single element from each isomorphism class of finite partial orders. Together with the fact that $W(c)$ is finite for each finite~$c$, this yields the claim. For the converse direction, let~$W$ be any finite monotone $\po$-dilator. To view finite partial orders as sequences, we fix a bijection
\begin{equation*}
\{0,\ldots,\len(c)-1\}\ni i\mapsto c[i]\in c
\end{equation*} 
with $\len(c)\in\mathbb N$ for each~$c\in\po_0$. We then pick an $n\in\mathbb N$ such that $(c,\rho)\in\tr(W)$ entails $\len(c)\leq n$. Given a well partial order~$X$, we consider
\begin{equation*}
\tr(W)+X:=\{(0,\xi)\,|\,\xi\in\tr(W)\}\cup\{(1,x)\,|\,x\in X\}
\end{equation*} 
with the discrete order on~$\tr(W)$ and the usual partial order on the sum, so that the only strict inequalities are $(1,x)<(1,x')$ for $x<x'$ in~$X$. Since this yields a well partial order, it suffices to construct a quasi embedding
\begin{equation*}
g:W(X)\to(\tr(W)+X)^{n+1}.
\end{equation*}
As we have seen in the introduction, each element of $W(X)$ has a unique normal form ${\sigma\nf W(\en_a^X)(\sigma_0)}$ with $(|a|,\sigma_0)\in\tr(W)$, for a certain embedding $\en_a^X$ that maps $|a|\in\po_0$ onto~$a\subseteq X$. When $\sigma\in W(X)$ has normal form as given, we put
\begin{equation*}
g(\sigma):=(y_0,\ldots,y_n)\quad\text{with}\quad y_i:=\begin{cases}
\big(1,\en_a^X(|a|[i])\big) & \text{when $i<\len(|a|)$},\\
\big(0,(|a|,\sigma_0)\big) & \text{otherwise}.
\end{cases}
\end{equation*}
Let us note that $y_n$ is always determined by the second case, due to the choice of~$n$. Given ${g(\sigma)\leq g(\tau)}$, we thus obtain $\sigma\nf W(\en_a^X)(\rho)$ and $\tau\nf W(\en_b^X)(\rho)$ for a single~$\rho$ and with ${|a|=|b|}$. The components with index $i<\len(|a|)$ ensure that we have $\en_a^X\leq\en_b^X$. Now monotonicity yields $\sigma\leq\tau$, as required.
\end{proof}

To connect with the original Kruskal theorem, we define certain collections of finite ordered trees. In the following, $l\star\langle t_0,\ldots,t_{k-1}\rangle$ stands for the tree with root label~$l$ and immediate subtrees $t_i$ (which recursively yields labels at all vertices). Note that $m=1$ amounts to the case without labels (as a single label~$l=0$ has no~effect). For $n=\omega$, the condition~$k<n$ is equivalent to~$k\in\mathbb N$.

\begin{definition}\label{def:trees}
For $m\in\mathbb N$ and $n\in\mathbb N\cup\{\omega\}$, we declare that $\mathbb T^m_n$ is recursively generated as follows: Whenever $t_0,\ldots,t_{k-1}\in\mathbb T^m_n$ with $k<n$ are already constructed, we add a new element $l\star\langle t_0,\ldots,t_{k-1}\rangle\in\mathbb T^m_n$ for every~$l<m$. Let us put~$\mathbb T_n:=\mathbb T^1_n$.
\end{definition}

An embedding of a tree~$t$ into~$t'$ will either send the root to the root and immediate subtrees into corresponding subtrees or it will send all of~$t$ into a proper subtree of~$t'$. We also demand that labels are preserved. The embeddability relation~$\tl$ can thus be characterized as follows (where we write $[k]:=\{0,\ldots,k-1\}$).

\begin{definition}\label{def:tree-embedding}
For trees $l\star\tau,l'\star\tau'\in\mathbb T^m_n$ with both $\tau=\langle t_0,\ldots,t_{j-1}\rangle$ and $\tau'=\langle t'_0,\ldots,t'_{k-1}\rangle$, we recursively declare
\begin{align*}
l\star\tau\tl l'\star\tau'\quad&:\Leftrightarrow\quad(l=l'\text{ and }\tau\tl^\star\tau')\text{ or }l\star\tau\tl t'_i\text{ for some }i<k,\\
\tau\tl^\star\tau'\quad&:\Leftrightarrow\quad\begin{cases}
\text{there is a strictly increasing $f:[j]\to[k]$}\\
\text{with $t_i\tl t'_{f(i)}$ for all $i<j$}.
\end{cases}
\end{align*}
\end{definition}

We come to the main result of this section. The condition that branchings are bounded will be removed in a corollary below.

\begin{theorem}\label{thm:fin-uniform-original}
The following are equivalent over $\rca_0$:
\begin{enumerate}[label=(\roman*)]
\item the uniform Kruskal theorem for finite monotone $\po$-dilators,
\item the original Kruskal theorem for finite ordered trees with bounded branchings, which is the statement that $(\mathbb T_n,\tl)$ is a well partial order for all~$n\in\mathbb N$.
\end{enumerate}
\end{theorem}
\begin{proof}
To show that~(i) implies~(ii), we fix $n\in\mathbb N$ and put
\begin{equation*}
W(X):=\{\langle x_0,\ldots,x_{k-1}\rangle\,|\,x_i\in X\text{ for }i<k<n\}.
\end{equation*}
When $X$ is a partial order, we declare that $\langle x_0,\ldots,x_{j-1}\rangle\leq_{W(X)}\langle x'_0,\ldots,x'_{k-1}\rangle$ holds precisely if there is a strictly increasing $f:[j]\to[k]$ with $x_i\leq_X x'_{f(i)}$ for all~$i<j$. Let us define $W(f)$ and $\supp_X$ analogous to the proof of Proposition~\ref{prop:monotone-to-wpo}, to obtain a finite monotone $\po$-dilator that is clearly normal. It suffices to construct a quasi embedding of $\mathbb T_n$ into the initial Kruskal fixed point $\mathcal TW$ of~$W$. As seen in the introduction, the latter comes with a map $\kappa:W(\mathcal TW)\to\mathcal TW$. We consider
\begin{equation*}
f:\mathbb T_n\to\mathcal TW\quad\text{with}\quad f(0\star\langle t_0,\ldots,t_{k-1}\rangle):=\kappa(\langle f(t_0),\ldots,f(t_{k-1})\rangle),
\end{equation*}
which amounts to a recursion over trees. A straightforward induction over the number of vertices confirms that $f(t)\leq_{\mathcal T W} f(t')$ entails $t\tl t'$, as the disjuncts in the definition of~$\tl$ correspond to those in Definition~\ref{def:Kruskal-fp}. In order to prove that~(ii) implies~(i), we first show that it entails
\begin{enumerate}[label=(\roman*')]\setcounter{enumi}{1}
\item the relation $\tl$ is a well partial order on $\mathbb T^m_n$ for all~$m,n\in\mathbb N$.
\end{enumerate}
The idea is to simulate labels $l<m$ by trees $t(l)$ that are incomparable under~$\tl$. Write \mbox{$t^j_k\in\mathbb T_{k+1}$} for the full $k$-branching tree of height~$j$, given by $t^0_k:=0\star\langle\rangle$ and $t^{j+1}_k:=0\star\langle t_0,\ldots,t_{k-1}\rangle$ with $t_i=t^j_k$ for all~$i<k$. A straightforward induction shows that $t^i_k\tl t^j_l$ with $i,k>0$ entails $i\leq j$ and~$k\leq l$. Our incomparable trees can thus be given by $t(l):=t_{l+1}^{m-l}$ for~$l<m$. Aiming at~(ii'), we may assume~$m<n$, which ensures $t(l)\in\mathbb T_{n+1}$. Let us now define $g:\mathbb T^m_n\to\mathbb T_{n+1}$ by
\begin{equation*}
g(l\star\langle t_0,\ldots,t_{k-1}\rangle):=0\star\langle t'_0,\ldots,t'_{n-1}\rangle\quad\text{with}\quad t'_i:=\begin{cases}
g(t_i) & \text{for }i<k,\\
t(l) & \text{for }k\leq i<n.
\end{cases}
\end{equation*}
To show that $g(s)\tl g(t)$ entails $s\tl t$, we write $s=l\star\sigma$ and $t=l'\star\tau$ with $\sigma=\langle s_0,\ldots,s_{j-1}\rangle$ and~$\tau=\langle t_0,\ldots,t_{k-1}\rangle$. First assume that $g(s)\tl g(t)$ holds by the first disjunct from Definition~\ref{def:tree-embedding}. We then have $s'_i\tl t'_i$ for all $i<n$, with $s'_i$ and $t'_i$ as in the definition of~$g$. As the definition of~$\mathbb T^m_n$ ensures $j,k<n$, we can deduce $t(l)=s'_{n-1}\tl t'_{n-1}=t(l')$ and hence~$l=l'$. For $i<\min(j,k)$, we also obtain $g(s_i)=s'_i\tl t_i'=g(t_i)$, so that we may inductively infer~$s_i\tl t_i$. To get $\sigma\tl^\star\tau$ and with that~$s\tl t$, we need only show~$j\leq k$. If the latter was false, we would get
\begin{equation*}
g(s_k)=s'_k\tl t'_k=t(l)=t_{l+1}^{m-l}.
\end{equation*}
However, as $g(s_k)$ is $n$-branching with $n>m\geq l+1$, it is straightforward to refute $g(s_k)\tl t_{l+1}^i$ by induction on~$i$. If $g(s)\tl g(t)$ holds by the second disjunct from Definition~\ref{def:tree-embedding}, we have $g(s)\tl t'_i$ for some $i<n$. Given that $g(s)\tl t(l)$ was just refuted, we must have $i<k$ and $t'_i=g(t_i)$. So we inductively get $s\tl t_i$ and then $s\tl t$. To complete the proof, we show that~(ii') entails~(i). Given a finite monotone $\po$-dilator~$W$ that is normal, we fix an injection $e:\tr(W)\to[m]$ for a suitable~$m\in\mathbb N$. Recall the bijections $[\len(c)]\ni i\mapsto c[i]\in c$ for $c\in\po_0$ that were fixed in the proof of Proposition~\ref{prop:monotone-to-wpo}. As in the latter, each finite $a\subseteq\mathcal TW$ gives rise to an embedding~$\en^{\mathcal TW}_a$ with domain $|a|\in\po_0$ and range $a\subseteq\mathcal TW$. Let us pick an $n\in\mathbb N$ such that $(c,\rho)\in\tr(W)$ entails $\len(c)<n$. From the paragraph after Definition~\ref{def:trace}, we recall the notation $\circ(a,\sigma)$ for elements of~$\mathcal TW$ as well as the height function~$h:\mathcal TW\to\mathbb N$. The latter justifies the recursive definition of a function $j:\mathcal TW\to\mathbb T^m_n$ with
\begin{equation*}
j\big(\circ(a,\sigma)\big):=e(|a|,\sigma)\star\langle t_0,\ldots,t_{\len(|a|)-1}\rangle\quad\text{for}\quad t_i:=j\left(\en_a^{\mathcal TW}(|a|[i])\right).
\end{equation*}
In view of~\cite[Definition~3.4]{frw-kruskal}, this can be cast as a recursion over subterms, which permits formalization in~$\rca_0$. It remains to show that $j(s)\tl j(t)$ entails~$s\leq_{\mathcal TW}t$, for which we use induction on~$h(s)+h(t)$. We write $s=\circ(a,\sigma)$ and $t=\circ(b,\tau)$. First, we assume that the inequality $j(s)\tl j(t)$ holds by the first disjunct in Definition~\ref{def:tree-embedding}, so that $e(|a|,\sigma)=e(|b|,\tau)$ and $j\circ\en_a^{\mathcal TW}(|a|[i])\tl j\circ\en_b^{\mathcal TW}(|b|[i])$ for~\mbox{$i<\len(|a|)=\len(|b|)$}. Inductively, we get $\en_a^{\mathcal TW}\leq\en_b^{\mathcal TW}$. Due to $\sigma=\tau$, we can invoke monotonicity and the equivalence from Definition~\ref{def:Kruskal-fp} to conclude
\begin{equation*}
s=\circ(a,\sigma)=\kappa\left(W\left(\en_a^{\mathcal TW}\right)(\sigma)\right)\leq_{\mathcal TW}\kappa\left(W\left(\en_b^{\mathcal TW}\right)(\tau)\right)=\circ(b,\tau)=t.
\end{equation*}
Note that the equalities hold by the definition of the notation $\circ(a,\sigma)$. If $j(s)\tl j(t)$ holds by the second disjunct from Definition~\ref{def:tree-embedding}, we have $j(s)\tl j(\en_b^{\mathcal TW}(|b|[i]))$ for some $i<\len(|b|)$. Given that $(|b|,\tau)\in\tr(W)$ entails $\supp_{|b|}(\tau)=|b|$, we can use the induction hypothesis to get
\begin{equation*}
s\leq_{\mathcal TW}\en_b^{\mathcal TW}(|b|[i])\in\left[\en_b^{\mathcal TW}\right]^{<\omega}\circ{\supp_{|b|}}(\tau)={\supp_{\mathcal TW}}\circ W\left(\en_b^{\mathcal TW}\right)(\tau).
\end{equation*}
We get $s\leq_{\mathcal TW}t$ by the second disjunct in the equivalence from Definition~\ref{def:Kruskal-fp}.
\end{proof}

Due to the bound on branchings, statement~(ii) of Theorem~\ref{thm:fin-uniform-original} can in fact be deduced from a result of Higman~\cite{higman52} (see also the survey by M.~Pouzet~\cite{Pouzet1985}). What~is often called Higman's lemma is a~weaker special case of this result. Kruskal~\cite{kruskal60} has shown that the result remains~valid when the bound on branchings is removed. Perhaps surprisingly, the extension by Kruskal does not have higher logical strength (in the case of trees without labels): Rathjen and Weiermann have shown that~(ii) is equivalent to statement~(iii) below (see~\cite[Corollary~2.1]{rathjen-weiermann-kruskal} and~\cite[Chapter~2]{meeren-thesis}). We keep the base theory~$\mathsf{ACA}_0$ from~\cite{rathjen-weiermann-kruskal}, even though~$\rca_0$ seems to suffice (see the first line of~\cite[Section~2]{rathjen-weiermann-kruskal} as well as~\cite[Section~6]{freund-commitment}).

\begin{corollary}\label{cor:unif-fin-trees}
Over $\mathsf{ACA}_0$, statement~(i) of Theorem~\ref{thm:fin-uniform-original} is equivalent to
\begin{enumerate}[label=(\roman*)]\setcounter{enumi}{2}
\item the original Kruskal theorem for the collection of all finite trees, which is the statement that $(\mathbb T_\omega,\tl)$ is a well partial order.
\end{enumerate}
\end{corollary}

Parallel to the equivalence between~(ii) and~(iii), one might believe that the uniform Kruskal theorem has similar strength for finite and for infinite trace. This, however, is firmly refuted by the equivalence with $\Pi^1_1$-comprehension in Theorem~\ref{thm:unif-kruskal} (which is proved in~\cite{frw-kruskal}). The stark contrast between the finite and infinite case can be explained in terms of quantifier complexity: Girard has shown that the notion of dilator on linear orders is~$\Pi^1_2$-complete (see the proof due to D.~Normann presented in \cite[Annex~8.E]{girard-book-part2}). We expect that the same is true for $\wpo$-dilators, even though details have not been verified. For finite $\po$-dilators, on the other hand, preservation of well partial orders is equivalent to monotonicity, which is an arithmetical condition (see Propositions~\ref{prop:wpo-monotone} and~\ref{prop:monotone-to-wpo} above as well as~\cite[Lemma~5.2]{frw-kruskal}). This~observation alone entails that the uniform Kruskal theorem for finite monotone $\po$-dilators must be strictly weaker than the principle of \mbox{$\Pi^1_1$-}comprehension. We conclude the section with a slight improvement of Theorem~\ref{thm:unif-kruskal} that was promised above:

\begin{remark}\label{rmk:ads}
In~\cite{frw-kruskal}, the equivalence from Theorem~\ref{thm:unif-kruskal} has been established over $\rca_0$ extended by the chain antichain principle. We argue that the latter may be weakened to~$\ads$. For a well partial order~$Z$, put $W_Z(X):=1+Z\times X$ and extend this into a normal $\po$-dilator. Details can be found in Examples~2.3 and~3.3~of~\cite{frw-kruskal}. In the cited paper, the chain antichain principle was assumed to ensure that $W_Z$ is a $\wpo$-dilator, so that the uniform Kruskal theorem could be applied. Working in $\rca_0+\ads$, one can show that $W_\alpha$ is a $\wpo$-dilator whenever~$\alpha$ is a well order: Given an infinite sequence $(\alpha_0,x_0),(\alpha_1,x_1),\ldots$ in $\alpha\times X$, the crucial task is to find an infinite subsequence in which the components $\alpha_i$ are weakly increasing. If~there are only finitely many different~$\alpha_i$, we can conclude by the infinite pigeonhole principle, which follows from~$\ads$ (via~\cite[Proposition~4.5]{hirschfeldt-shore}). Otherwise we achieve $\alpha_0<_{\mathbb N}\alpha_1<_{\mathbb N}\ldots$ by passing to a subsequence. We may then form $A:=\{\alpha_i\,|\,i\in\mathbb N\}$ in~$\rca_0$. By $\ads$ we find $f:\mathbb N\to A$ such that $i<j$ entails $f(i)\prec f(j)$ with respect to the order on~$\alpha$. We can now search for $i(0)<i(1)<\ldots$ and \mbox{$j_0<j_1<\ldots$} with $\alpha_{i(k)}=f(j_k)$, which yields $\alpha_{i(0)}\prec\alpha_{i(1)}\prec\ldots$ as desired. Due to Example~3.9 of~\cite{frw-kruskal}, the initial Kruskal fixed point~$\mathcal TW_\alpha$ coincides with $\alpha^{<\omega}$, the set of finite sequences in~$\alpha$ with the order from Higman's lemma. So over $\rca_0+\ads$, the uniform Kruskal theorem for $\wpo$-dilators entails that $\alpha^{<\omega}$ is a well partial order for any well order~$\alpha$. The latter entails arithmetical comprehension and in particular the chain antichain principle, via a result of Girard and Hirst (see~\cite[Section~II.5]{girard87} and~\cite[Theorem~2.6]{hirst94}). Thus the base theory of Theorem~\ref{thm:unif-kruskal} can be weakened as promised. Note that $\tr(W_\alpha)$ is infinite when the same holds for~$\alpha$. The chain antichain principle cannot be derived from the uniform Kruskal theorem for finite monotone $\po$-dilators, as the latter is a $\Pi^1_1$-statement.
\end{remark}

\section{On Kruskal fixed points in the absence of~$\mathsf{ADS}$}\label{sect:without-ADS}

In this section, we show that the class of $\wpo$-dilators is heavily restricted in systems that reject the axiom $\ads$. Moreover, in Theorems~\ref{thm:dichotomy} and~\ref{thm:main}, we prove that a certain variant of the uniform Kruskal theorem is weak over $\rca_0$, while it entails $\Pi^1_1$-comprehension in the presence of~$\cac$. The following notion will play a crucial role.

\begin{definition}
	A $\po$-dilator~$W$ is called unary if $\supp_X(\sigma)$ has at most one element, for every partial order~$X$ and all~$\sigma\in X$.
\end{definition}

Note that the following result involves a theory that is not sound. Normal $\wpo$-dilators that are not unary do in fact exist (see, e.\,g., the previous section). Later, the result will be applied inside $\omega$-models of $\neg\ads$.

\begin{lemma}\label{lem:ads-unary}
	In the theory $\rca_0+\neg\ads$, one can derive that any normal~$\wpo$-dilator~$W$ is unary.
\end{lemma}
\begin{proof}
	Towards a contradiction, assume that $a:=\supp_X(\sigma)$ contains two different elements~$x_0$ and~$x_1$. We invoke~$\neg\ads$ to pick an infinite well order~$Z$ that has no infinite ascending sequence, i.\,e., such that the inverse order~$Z^*$ is also well founded. Let us consider the well partial order
	\begin{equation*}
		Y:=Z+Z^*+a:=\{(i,y)\,|\,\text{either $i\in\{0,1\}$ and $y\in Z$ or $i=2$ and $y\in a$}\},
	\end{equation*}
	where the only inequalities are $(0,z)\leq_Y(0,z')$ and $(1,z)\geq_Y(1,z')$ for~$z\leq_Z z'$ as well as $(2,x)\leq_Y (2,x)$ for $x\in a$ (one could also use the induced order on~$a\subseteq X$). For each $z\in Z$, we have a quasi embedding
	\begin{equation*}
		f_z:a\to Y\quad\text{with}\quad f_z(x):=\begin{cases}
			(0,z) & \text{if }x=x_0,\\
			(1,z) & \text{if }x=x_1,\\
			(2,x) & \text{otherwise}.
		\end{cases}
	\end{equation*}
	Due to the support condition from Definition~\ref{def:po-dilator}, we may write $\sigma=W(\iota)(\sigma_0)$ with $\sigma_0\in W(a)$, where $\iota:a\hookrightarrow X$ is the inclusion. The naturality of supports ensures that we have
	\begin{equation*}
		[\iota]^{<\omega}\circ{\supp_a(\sigma_0)}={\supp_X}\circ W(\iota)(\sigma_0)=\supp_X(\sigma)=a
	\end{equation*}
	and hence $\supp_a(\sigma_0)=a$. For $z\in Z$ we put $\tau_z:=W(f_z)(\sigma_0)\in W(Y)$ and note
	\begin{align*}
		\supp_Y(\tau_z)&={\supp_Y}\circ W(f_z)(\sigma_0)=[f_z]^{<\omega}\circ{\supp_a}(\sigma_0)=\\
		{}&=\big\{(i,y)\,|\,\text{either $i\in\{0,1\}$ and $y=z$ or $i=2$ and $y\in a\backslash\{x_0,x_1\}$}\big\}.
	\end{align*}
	Given that~$Z$ is infinite, we have an injection~$\mathbb N\to Z$ (indeed $Z\subseteq\mathbb N$ will be given as a set of codes). Let us compose the latter with the map $z\mapsto \tau_z$, to get an infinite sequence in~$W(Y)$. The latter is a well partial order, as the same holds for~$Y$ and since $W$ is a $\wpo$-dilator. So~$Z$ must contain some $y\neq z$ with~${\tau_y\leq_{W(Y)} \tau_z}$. By~the normality condition from Definition~\ref{def:po-dilator}, we know that the elements ${(i,y)\in\supp_Y(\tau_y)}$ must be bounded by some elements of $\supp_Y(\tau_z)$, which can only be true if we have $(i,y)\leq_Y(i,z)$ for each $i\in\{0,1\}$. However, this yields $y\leq_Z z$ as well as $y\geq_Z z$, in contradiction with~$y\neq z$.
\end{proof}

Our next result confirms that being unary is a serious restriction (cf.~Theorem~\ref{thm:unif-kruskal}). We recall that $1$ denotes the order with a single element.

\begin{lemma}\label{lem:unary-aca}
	The following are equivalent over~$\rca_0$:
	\begin{enumerate}[label=(\roman*)]
		\item the uniform Kruskal theorem for unary monotone~$\po$-dilators~$W$ with the property that $W(1)$ is a well partial order,
		\item the principle of arithmetical comprehension.
	\end{enumerate}
\end{lemma}
\begin{proof}
The implication from~(i) to~(ii) holds by an argument from~\cite{frw-kruskal}, which has been recalled in Remark~\ref{rmk:ads} above. In contrast with the remark, our implication does not rely on $\ads$, because~(i) does not require that $W$ is a $\wpo$-dilator, even though this follows over a sufficient base theory. For the implication from (ii) to~(i), we consider the partial order
\begin{equation*}
Y:=W(0)+W(1)=\{(i,\sigma)\,|\,i\in\{0,1\}\text{ and }\sigma\in W(i)\}
\end{equation*}
in which the two summands are incomparable, so that $(i,\sigma)\leq_Y(i,\tau)$ for $\sigma\leq_{W(i)}\tau$ are the only inequalities. Given that $W(1)$ is a well partial order, the same holds for~$Y$, as the empty function $\iota:0\hookrightarrow 1$ induces an embedding ${W(\iota):W(0)\to W(1)}$. Assuming~(ii), we get Higman's lemma in the form that makes a well partial order out of $Y^{<\omega}$ (see, e.\,g., \cite[Theorem~X.3.22]{simpson09}). For unary~$W$, elements of~$\tr(W)$ have the form~$(i,\sigma)$ with $\sigma\in W(i)$ for~$i\in\{0,1\}$. We can thus consider
	\begin{equation*}
		j:\mathcal T W\to Y^{<\omega}\quad\text{with}\quad j\big(\circ(a,\sigma)\big):=\begin{cases}
			\langle(0,\sigma)\rangle & \text{if }a=\emptyset,\\
			\langle(1,\sigma)\rangle^\frown j(t) & \text{if }a=\{t\}.
		\end{cases}
	\end{equation*}
Closely parallel to the proof of Theorem~\ref{thm:fin-uniform-original}, one inductively shows that $j$ is a quasi embedding. It follows that $\mathcal T W$ is a well partial order, as needed for the uniform~Kruskal~theorem in~(i).
\end{proof}

In order to prove that the uniform Kruskal theorem is weak over~$\rca_0$, one might try to find an $\omega$-model of~$\neg\ads$ that validates statement~(i) above. However, no such model can exist, as the equivalent statement~(ii) entails $\ads$. In the following we look at $\Pi^1_1$-consequences of~(i), which are guaranteed to reflect into $\omega$-models. Let us recall that $\po$-dilators are officially coded as subsets of~$\mathbb N$. The powerset of $\mathbb N$ will be denoted by $\mathcal P(\mathbb N)$.

\begin{theorem}\label{thm:main-general}
Consider a theory $\mathsf T\supseteq\aca_0$ and a class $\Gamma\subseteq\mathcal P(\mathbb N)$ defined by a $\Sigma^1_1$-formula. Assume $\mathsf T$ proves that $W(1)$ is a well partial order for any unary monotone $\po$-dilator $W\in\Gamma$. Let us write $\varphi$ for the uniform Kruskal theorem for $\wpo$-dilators in $\Gamma$. Then any $\Pi^1_2$-theorem of $\rca_0+\varphi$ is provable in~$\mathsf T$.
\end{theorem}
\begin{proof}
Assume $\rca_0+\varphi$ proves $\forall X\subseteq\mathbb N.\,\theta(X)$ for a $\Sigma^1_1$-formula~$\theta$. To conclude by completeness, we consider an arbitrary model $\mathcal M\vDash\mathsf T$ and show that $\mathcal M\vDash\theta(X)$ holds for any $X$ in its second order part. Let $\mathcal N\subseteq\mathcal M$ be the $\omega$-submodel of all sets that $\mathcal M$ believes to be computable in~$X$. We recall that we have $\mathcal N\vDash\rca_0$ (see, e.\,g., \cite[Theorem~VIII.1.3]{simpson09}). A classical result of S.~Tennen\-baum asserts that $\ads$ fails in the realm of computable sets. The usual proof of this result (see, e.\,g., the textbook by J.~Rosenstein~\cite{Rosenstein82}) relativizes and is readily formalized in $\aca_0\subseteq\mathsf T$ (see Lemma~\ref{lem:Tennenbaum} below for a less direct argument over~$\rca_0$). We thus get $\mathcal N\vDash\neg\ads$. In view of Proposition~\ref{prop:wpo-monotone} and Lemma~\ref{lem:ads-unary}, we can infer $\mathcal N\vDash\varphi_0\to\varphi$, where $\varphi_0$ is the uniform Kruskal theorem for unary and monotone $\po$-dilators in~$\Gamma$. By the previous lemma we have $\mathcal M\vDash\varphi_0$. Being a unary monotone $\po$-dilator is an arithmetical property, as it is determined on the subcategory of finite partial orders. For more details on this point we refer to~\cite{frw-kruskal}, where it is also shown that the initial Kruskal fixed point $\mathcal T W$ is computable in~$W$. Given that $\Gamma$ is $\Sigma^1_1$, it follows that $\varphi_0$ is $\Pi^1_1$, so that we get $\mathcal N\vDash\varphi_0$ and hence $\mathcal N\vDash\varphi$. By~assumption we get $\mathcal N\vDash\theta(X)$ and then $\mathcal M\vDash\theta(X)$, as $\theta$ is $\Sigma^1_1$.
\end{proof}

Part~(a) of the following dichotomy result is obtained as a direct application of the previous theorem. With some additional work, one can replace $\aca_0$ by a weaker theory, as we shall see below. To indicate further applications of the previous theorem, we note that one can take~$\Gamma$ to be the class of $\po$-dilators~$W$ such that $W(1)$ admits a quasi embedding into a fixed~$X$ that $\mathsf T$ proves to be a well~partial order.

\begin{theorem}\label{thm:dichotomy}
Write $\psi$ for the uniform Kruskal theorem for $\wpo$-dilators $W$ such that $W(1)$ is finite.
\begin{enumerate}[label=(\alph*)]
\item Any $\Pi^1_2$-theorem of $\rca_0+\psi$ is provable in~$\aca_0$.
\item In $\rca_0+\cac$ one can prove that $\psi$ is equivalent to $\Pi^1_1$-comprehension.
\end{enumerate}
\end{theorem}
\begin{proof}
To obtain~(a), one applies the previous theorem to the class~$\Gamma$ of $\po$-dilators $W$ with the property that $W(1)$ is finite. The latter is a $\Sigma^1_1$-condition (in fact arithmetical) that entails that $W(1)$ is a well partial order. In order to derive~(b) from Theorem~\ref{thm:unif-kruskal} (originally proved in~\cite{frw-kruskal}), we only need to show that a normal $\wpo$-dilator~$W$ can be transformed into a normal $\wpo$-dilator~$W'$ such that $W'(1)$ is finite and we have a quasi embedding $j:\mathcal T W\to\mathcal T W'$ between the initial Kruskal fixed points. For a partial order $X$, we set
		\begin{equation*}
				W'(X) := \{\star,+\} \cup \{\langle x, y, \sigma \rangle \in X \times X \times W(X) \mid x \neq y\}.
		\end{equation*}
		Here $\star$ and $+$ represent new elements that are incomparable with any others, while we have
\begin{equation*}
\langle x,y,\sigma\rangle\leq_{W'(X)}\langle x',y',\sigma'\rangle\quad\Leftrightarrow\quad x\leq_X x'\text{ and }y\leq_X y'\text{ and }\sigma\leq_{W(X)}\sigma'.
\end{equation*}
To get a functor, we recall that any quasi embedding $f:X\to Y$ is injective. Now, this allows us to put
\begin{equation*}
W'(f)(\star) := \star,\quad W'(f)(+) :=+,\quad W'(f)(\langle x, y, \sigma \rangle) := \langle f(x), f(y), W(f)(\sigma) \rangle. 
\end{equation*}
Let us also define functions $\supp'_X: W'(X) \to [X]^{<\omega}$ by setting
		\begin{equation*}
			\supp'_X(\star):=\emptyset,\quad\supp'_X(+) := \emptyset,\quad\supp'_X(\langle x, y, \sigma \rangle) := \{x, y\} \cup \supp_X(\sigma),
		\end{equation*}
where $\supp_X:W(X)\to[X]^{<\omega}$ is the support function already provided by the $\wpo$-dilator~$W$. One readily checks that this turns $W'$ into a normal \mbox{$\po$-dilator}. Crucially, $\cac$ ensures that the well partial orders are closed under products (see~\cite{cholak-RM-wpo}), which implies that $W'$ is a $\wpo$-dilator. Also note that $W'(1)=\{\star,+\}$ is finite. To construct our quasi embedding $j:\mathcal TW\to\mathcal TW'$, we first note that $\mathcal TW'$ contains distinct elements $\overline\star:=\circ(\emptyset,\star)$ and $\overline +:=\circ(\emptyset,+)$. As explained in the introduction, each $\circ(a, \sigma) \in \tw$ gives rise to an element ${W(\en^{\tw}_a)(\sigma) \in W(\tw)}$. Recall that the Kruskal fixed point of $W'$ comes with a bijection ${\kappa': W'(\tw') \to \tw'}$. As in the proof of Theorem~\ref{thm:fin-uniform-original}, we can use recursion over terms to define
		\begin{equation*}
			j(\circ(a, \sigma)) := \kappa'\left(\left\langle\overline\star,\overline +, W(j \circ \en^{\mathcal T W}_a)(\sigma)\right\rangle\right)\period
		\end{equation*}
To be more precise, we note that the finite function $j \circ \en^{\mathcal T W}_a$ is available in the recursion step. In the inductive proof below, the induction hypothesis will imply that $j \circ \en^{\mathcal T W}_a$ is a quasi embedding, which is needed to ensure that $W(j \circ \en^{\mathcal T W}_a)$ is defined. To disentangle the recursion and the induction, one can assign a default value such as $j(\circ(a,\sigma)):=\overline\star$ for the hypothetical case that $j \circ \en^{\mathcal T W}_a$ is no quasi embedding (which is decidable). As preparation for the aforementioned induction, we define a length function $l:\mathcal TW\to\mathbb N$ with $l(\circ(a,\sigma)):=1+\sum_{t\in a}l(t)$, by another recursion over terms. To check that $j(s)\leq_{\mathcal T W'}j(t)$ entails $s\leq_{\mathcal T W}t$, one argues by induction on $l(s)+l(t)$. For $s=\circ(a,\sigma)$ and $t=\circ(b,\tau)$, the induction hypothesis ensures that $j$ is a quasi embedding on $a\cup b$, due to our choice of length function. Besides this observation, the only noteworthy point in the induction is that we have
\begin{equation*}
\supp'_{\mathcal T W'}\left(\left\langle\overline\star,\overline +,W(j \circ \en^{\mathcal T W}_b)(\tau)\right\rangle\right)=\{\overline\star,\overline +\}\cup[j]^{<\omega}(b).
\end{equation*}
To avoid that the equivalence from Definition~\ref{def:Kruskal-fp} yields $j(s)\leq_{\mathcal T W'}j(t)$ for any~$t$, we need to ensure that $j(s)\not\leq_{\mathcal T W'}x$ holds for $x\in\{\overline\star,\overline +\}$. The latter can itself be deduced from the equivalence in Definition~\ref{def:Kruskal-fp}, as $\star$ and $+$ have empty supports and are incomparable with other elements.
\end{proof}

According to Theorem~\ref{thm:fin-uniform-original}, the uniform Kruskal theorem for finite monotone $\po$-dilators entails the original Kruskal theorem for trees with bounded branchings, provably in~$\rca_0$ (see Corollary~\ref{cor:unif-fin-trees} for the case with arbitrary branchings). Over a sufficient base theory, the finite monotone $\po$-dilators coincide with the finite $\wpo$-dilators, by Propositions~\ref{prop:wpo-monotone} and~\ref{prop:monotone-to-wpo}. The following implies that the resulting versions of the uniform Kruskal theorem differ over~$\rca_0$.

\begin{corollary}
If $\varphi$ is a $\Sigma^1_2$-theorem of~$\aca_0$, then $\rca_0+\varphi$ does not show that the uniform Kruskal theorem for finite $\wpo$-dilators entails the original Kruskal theorem for binary trees, i.\,e., the statement that $(\mathbb T_3,\tl)$ is a well partial order (cf.~Definitions~\ref{def:trees} and~\ref{def:tree-embedding}).
\end{corollary}
\begin{proof}
As shown by H.~Friedman, the original Kruskal theorem for finite trees with arbitrary branchings is unprovable in predicative theories (see~\cite{simpson85}). The version for binary trees is unprovable in $\aca_0$. In order to see this, one should first recall that $\aca_0$ is conservative over Peano arithmetic (see, e.\,g., \cite[Theorem~III.1.16]{hajek91}). The latter cannot prove the well foundedness of a certain term system $\varepsilon_0$ that represents its proof theoretic ordinal, due to classical work of G.~Gentzen~\cite{gentzen43} (see, e.\,g.,~\cite{buchholz91} for a modern presentation). Finally, the well foundedness of $\varepsilon_0$ follows from the original Kruskal theorem for binary trees, by unpublished work of D.~de Jongh (see~\cite{schmidt75} for the attribution and, e.\,g., \cite{freund-commitment} for a detailed proof). Now if $W$ is a finite $\po$-dilator, then $W(1)$ is finite. This holds because any $\sigma\in W(1)$ has a normal form $\sigma=W(\en_a^1)(\sigma_0)$ with $a\subseteq 1$ and $(|a|,\sigma_0)\in\tr(W)$, as explained in the introduction. Hence the uniform Kruskal theorem for finite $\wpo$-dilators is entailed by statement~$\psi$ from the previous theorem. So if the corollary was false, $\rca_0+\psi$ would prove that $\varphi$ implies the original Kruskal theorem for binary trees. This implication is $\Pi^1_2$, so that the previous theorem would make it provable in~$\aca_0$. As the latter proves~$\varphi$, it would thus prove the binary Kruskal theorem, which we have seen to be false.
\end{proof}
 
In the following, we show how to improve part~(a) of Theorem~\ref{thm:dichotomy}. To see that Tennenbaum's result about $\ads$ is available over~$\rca_0$, we combine some results from the literature. By $\varphi^X$ we denote the formula that results from~$\varphi$ when all second order quantifiers are restricted to the sets that are computable relative to~$X\subseteq\mathbb N$.

\begin{lemma}\label{lem:Tennenbaum}
We have $\neg\ads^X$ for all $X\subseteq\mathbb N$, provably in~$\rca_0$.
\end{lemma}
\begin{proof}
By work of D.~Hirschfeldt and R.~Shore~\cite{hirschfeldt-shore-RT} and their joint work with T.~Slaman~\cite{hirschfeldt-shore-slaman}, we have $\rca_0\vdash\ads\to\mathsf{HYP}$ and hence $\rca_0\vdash\ads^X\to\mathsf{HYP}^X$. Here $\mathsf{HYP}$ is the hyperimmunity principle which asserts that, for any~$Y\subseteq\mathbb N$, there is a function that is not dominated by any \mbox{$Y$-}computable function. It is clear that $\rca_0$ proves $\neg\mathsf{HYP}^X$, as needed.
\end{proof}

The following result characterizes the fragment of Lemma~\ref{lem:unary-aca} that is needed for Theorem~\ref{thm:dichotomy}. To explain the notation $\omega^{\omega^\omega}$ in part~(iii), we recall that $\omega^X$ denotes the lexicographic order on the non-increasing finite sequences in a given linear order~$X$.

\begin{lemma}\label{lem:omega3_kruskal_unary_finite}
	The following are equivalent over $\rca_0$:
	\begin{enumerate}[label=(\roman*)]
             \item the uniform Kruskal theorem for unary monotone $\po$-dilators~$W$ such that $W(1)$ is finite,
             \item Higman's lemma for sequences with entries from a finite partial order,
		\item the statement that $\omega^{\omega^\omega}$ is a well order.
	\end{enumerate}
\end{lemma}
\begin{proof}
Under the remaining conditions from~(i), the value $W(1)$ is finite precisely if the same holds for the trace~$\tr(W)$. Furthermore, Higman's lemma for finite orders is equivalent to the statement that $(\mathbb T^m_2,\tl)$ is a well partial order for any~${m\in\mathbb N}$, in the notation from the previous section. Thus the equivalence between~(i) and~(ii) holds by the proof of Theorem \ref{thm:fin-uniform-original} (also consider Remark~\ref{rmk:ads} with finite~$Z$). The~equivalence between (ii) and (iii) relies on the fact that the sequences in an order with $n+1$ elements have maximal order type~$\omega^{\omega^n}$, as shown by D.~de Jongh and R.~Parikh~\cite{deJongh-Parikh} (see, e.\,g.,~\cite{schuette-simpson,Hasegawa94} for the formalization in $\rca_0$).
\end{proof}

As promised, part~(a) of the following theorem improves part~(a) of Theorem~\ref{thm:dichotomy}. Concerning part~(b) below, we note that the infinite pigeonhole principle is a \mbox{$\Pi^1_1$-statement}, which one may include as premise of a $\Pi^1_2$-theorem as in~(a).

\begin{theorem}\label{thm:main}
	Let $\psi$ express the uniform Kruskal theorem for $\wpo$-dilators $W$ such that $W(1)$ is finite.
	\begin{enumerate}[label=(\alph*)]
		\item Any $\Pi^1_2$-theorem of $\rca_0+\psi$ is provable in $\rca_0+\text{``$\omega^{\omega^\omega}$ is well founded"}$.
		\item The well foundedness of $\omega^{\omega^\omega}$ is provable in the extension of $\rca_0$ by statement~$\psi$ and the infinite pigeonhole principle.
	\end{enumerate}
\end{theorem}
\begin{proof}
Part~(a) can be proved exactly like Theorem~\ref{thm:main-general}, since Lemmas~\ref{lem:Tennenbaum} and~\ref{lem:omega3_kruskal_unary_finite} ensure that the relevant ingredients remain available. In order to reduce part~(b) to the previous lemma, it suffices to observe that any unary monotone $\po$-dilator~$W$ with finite trace is a $\wpo$-dilator, over $\rca_0$ with the infinite pigeonhole principle. To confirm this, we consider a well partial order~$X$ and an infinite sequence $\sigma_0,\sigma_1,\ldots$ in $W(X)$. Each entry has a normal form~$\sigma_i=W(\en_{a(i)}^X)(\sigma_i')$ with $(|a(i)|,\sigma_i')\in\tr(W)$, as explained in the introduction. Given that $W$ has finite trace, the infinite pigeonhole principle allows us to assume that $(|a(i)|,\sigma_i')$ is independent of~$i$. Since $W$ is unary, the sets $a(i)\subseteq X$ contain at most one element. In the non-trivial case they are non-empty, so that we may write $a(i)=\{x_i\}$ with $x_i\in X$. Due to the fact that $X$ is a well partial order, we find $i<j$ with $x_i\leq_X x_j$. Monotonicity yields $\sigma_i\leq_{W(X)}\sigma_j$, as needed to make $W$ a $\wpo$-dilator.
\end{proof}

To conclude, we mention one further variation of our dichotomy result.

\begin{remark}\label{rmk:cons-rca}
Let $\psi'$ be the uniform Kruskal theorem for $\wpo$-dilators $W$ such that $W(1)$ coincides with $W(0)$, i.\,e., with $W(1)=\rng(W(\iota))$ for the empty inclusion $\iota:0=\emptyset\hookrightarrow 1$. If $W$ satisfies this condition and is unary, then all elements of its trace have the form $(\emptyset,\sigma)$. The latter entails $\mathcal TW\cong W(0)$, so that $\rca_0$ alone proves the corresponding instance of the uniform Kruskal theorem. It follows that $\rca_0+\psi'$ is a $\Pi^1_2$-conservative extension of~$\rca_0$. Over $\rca_0+\cac$, however, statement~$\psi'$ is still equivalent to $\Pi^1_1$-comprehension, by the proof of Theorem~\ref{thm:dichotomy}.
\end{remark}

\bibliographystyle{plain}
\bibliography{Uniform-Kruskal-finite}

\end{document}